\documentclass{article}

\usepackage{arxiv}

\usepackage[utf8]{inputenc} 
\usepackage[T1]{fontenc}    
\usepackage{hyperref}       
\usepackage{url}            
\usepackage{booktabs}       
\usepackage{nicefrac}       
\usepackage{microtype}      
\usepackage{lipsum}
\usepackage{graphicx}
\graphicspath{ {./images/} }
\usepackage{graphicx}
\usepackage{array}
\usepackage{multirow}
\usepackage{tabu}
\usepackage{xcolor}
\usepackage[english]{babel}
\usepackage{amsmath, amsthm, amssymb, amsfonts}\DeclareMathOperator{\Tr}{Tr}
\newtheorem*{remark}{Remark}
\newtheorem{prop}{Proposition}
\title{Numerical Bifurcation Analysis of Turing and Symmetry Broken Patterns of a Vegetation PDE Model}

\author{
 Konstantinos Spiliotis \\
  Mathematical Institute, University of Rostock\\
  Rostock, Germany\\
   \And
 Lucia Russo* \\
 Istituto di Scienze e Tecnologie \\ 
 per l'Energia e la Mobilit\'a  Sostenibili, CNR\\
              Naples, Italy\\
  *Corresponding author\\
  \texttt{lucia.russo@stems.cnr.it} \\
  \And 
  Francesco Giannino*\\
  Dipartimento di Agraria \\ Universit\'a degli Studi di Napoli Federico II \\
  Naples, Italy \\ *Corresponding author\\
  \texttt{francesco.giannino@unina.it }
 \And 
Constantinos Siettos\\Dipartimento di Matematica \\
e Applicazioni ``Renato Caccioppoli" \\Universit\'a degli Studi di Napoli Federico II \\ Naples, Italy\\
}

\begin{document}

\maketitle
\begin{abstract}
We study the mechanisms of pattern formation for vegetation dynamics in water-limited regions. Our analysis is based on a set of two partial differential equations (PDEs) of reaction–diffusion type for the biomass and water and one ordinary differential equation (ODE) describing the dependence of the toxicity on the biomass. We perform a linear stability analysis in the one-dimensional finite space, we derive analytically the conditions for the appearance of Turing instability that gives rise to spatio-temporal patterns emanating from the homogeneous solution, and provide its dependence with respect to the size of the domain. Furthermore, we perform a numerical bifurcation analysis in order to study the pattern formation of the inhomogeneous solution, with respect to the precipitation rate, thus analyzing the stability and symmetry properties of the emanating patterns. Based on the numerical bifurcation analysis, we have found new patterns, which form due to the onset of secondary bifurcations from the primary Turing instability, thus giving rise to a multistability of asymmetric solutions.
\end{abstract}



\section{Introduction}
It is well known, that the self-organized spatio-temporal patterning of vegetation, especially in water-limited regions, comes as a feedback response to ecosystem stability and species diversity \cite{vincenot2016spatial,zhao2019shaping,callaway2021belowground}.
Thus, the demystification of the mechanisms that pertain to the formation and dynamics of such spatio-temporal vegetation patterns is at the forefront of contemporary ecological and environmental research efforts\cite{kefi2007spatial,tarnita2017theoretical}. An important open research question revolves around the relation between vegetation patterning changes/disturbances and catastrophic/irreversible transitions, both in the environmental landscape and biodiversity. For example, K{\'e}fi et al. \cite{kefi2007spatial} showed that patch-size distributions in arid Mediterranean ecosystems may serve as early-warning signals for the onset of desertification. Bonanomi et al. \cite{bonanomi2014ring} suggested that vegetation rings facilitates the diversity of species. Zhao et al. \cite{zhao2019shaping} showed that the patchy vegetation in salt marsh ecosystems promotes species bio-diversity. Such patterns include but are not limited to stripes, spots, rings, labyrinth-like structures and spiral waves \cite{callaway2021belowground}.\par
To explain such self-organizing spatio-temporal patterns, various mathematical dynamical models have been used ranging from microscopic models, including stochastic cellular automata \cite{silvertown1992cellular,pascual2002cluster, kefi2007spatial}, and agent-based/individualistic models of deterministic ordinary differential equations 
 (ODEs) \cite{vincenot2016spatial,vincenot2017plant}, to continuum models of partial differential equations (PDEs) \cite{bonanomi2005negative,carteni2012negative,marasco2014vegetation,severino2017effects}. The keystone idea that underpins the above mathematical models is that of the ``scale-dependent feedback'' \cite{rietkerk2008regular} mechanism between species and limited resources. This mechanism is governed by the so-called activator–inhibitor principle introduced by Turing in his  celebrated 1952 paper ``The chemical basis of morphogenesis" \cite{turing1952, Tur90} on the spontaneous formation of patterns in diffusion-reaction systems (see also the discussion in \cite{maini1997spatial,rietkerk2008regular,ball2015forging,krause2021modern}).\par
What is usually done with such continuum-level vegetation reaction-diffusion PDEs, is temporal simulation and linear stability analysis  (see e.g. \cite{klausmeier1999regular,carteni2012negative,marasco2014vegetation,gowda2016assessing}) of the homogeneous (spatial independent) dynamics \cite{marasco2014vegetation,gowda2016assessing}. However, simple temporal simulations and/or linear stability analysis are not adequate for the investigation of far-from-the-equilibrium nonlinear phenomena. For example, in several studies it has been shown, that Turing instabilities may experience secondary bifurcations leading to far-from-equilibrium oscillating solutions \cite{lamb2006hopf,spiliotis2018analytical}, spatio-temporal chaos \cite{aragon2012nonlinear,banerjee2012turing} and symmetry-breaking bifurcations \cite{barrio2002size}. In such regimes, nonlinearities play a key role not only in stabilizing a pattern, but also in producing unsuspected bifurcations lined with catastrophic transitions \cite{russo2019bautin,spiliotis2018analytical,spiliotis2021analytical}. Thus, to systematically investigate such phenomena systematically, the exploitation of the full arsenal of numerical bifurcation theory is of out-most importance \cite{satnoianu2000turing,henderson2016alternative,russo2019bautin,spiliotis2021analytical}).\par 
Here, we construct the full bifurcation diagram of a  vegetation model consisting of two coupled PDEs describing the dynamics of plant biomass, water concentration according to \cite{klausmeier1999regular} and one ODE describing the dynamics of toxic compounds \cite{carteni2012negative}, with respect to the precipitation rate in the one dimensional finite domain. First, we provide analytical results for the location of Turing bifurcations, also with respect to the size of the domain, by performing a linear stability analysis, thus considering spatial-temporal perturbations of the homogeneous equilibrium state. Furthermore, we perform a numerical bifurcation analysis to track branches of both stable and unstable far-from-the-homogeneous equilibrium patterns, thus finding novel asymmetric patterns that arise due to secondary bifurcations of the initial Turing instability. 
This is the first time that such an analysis for such a vegetation model is provided, thus revealing regions of multi-stability and novel symmetric and far-from-the-homogeneous equilibrium asymmetric patterns.

\section{The mathematical model}
The mathematical model analyzed in this paper was proposed by Marasco et al.\cite{marasco2014vegetation} to simulate the dynamics of three state variables, namely, the biomass $B$, the soil water $W$, and the toxic compounds $T$. Indeed, the positive (of water) and negative (of toxicity) feedbacks on plant biomass can explain the occurrence of different vegetation patterns also in non water-limited environmental conditions.

The soil water $W$ ($kg/m^2$) increases uniformly due to  the rain precipitation $p$ and is reduced by the evaporation process at a rate $lW$ and plants transpiration at a rate $rB^2W$. Moreover, the water diffuses in the soil with a diffusion coefficient $D_W$. The plant biomass $B$ ($kg/m^2$) grows at a nonlinear rate $ rB^2W$ according to water availability in the soil and dies due to a natural rate $d$ and an extra loss induced by the presence of toxic compounds $T$. The intensity of toxicity depends on the plant sensibility, here parametrized by the parameter $s$. Plant lateral propagation is modelled by a dispersal term of a diffusion coefficient $D_B$.
Toxic compounds $T$ ($kg/m^2$) are produced by the dead biomass in a fraction $q$ and decay by the decomposition process with a rate $k$, while they are washed out via precipitation with a rate $w$. The lateral movement of $T$ is not considered, assuming that the toxic compounds do not move in the soil.
These processes are formalized by the following system of two PDEs and one ODE.
\begin{equation} 
\begin{split}
B_t & =D_BB_{xx}+cB^2W-(d+sT)B   \\
 W_t & =D_WW_{xx} +p-rB^2W-lW \\
   T_t & =q(d+sT)B-(k+wp)T ,
\end{split}
\label{eq:system}
\end{equation}
With Neumann boundary conditions, i.e.,:
\begin{equation}
 B_x(0,t)= B_x(L,t)=0,  W_x(0,t)= W_x(L,t)=0,  T_x(0,t)= T_x(L,t)=0.
 \label{eq:NBC}
\end{equation}
In this study, the main bifurcation parameter is the precipitation rate while the exact values of the other parameters are given in Table \ref{tab:parameters}
\begin{table}[h]
\centering
\begin{tabular}{ p{2cm}||p{6cm}|p{1cm} }
 \hline
 \hline
 parameter & Description & Values\\
 \hline \hline 
$c$ & Growth rate of biomass $B$  & 0.002\\
$d$ & Death rate of biomass $B$  & 0.01\\
$k$ &Decay rate of toxicity $T$  & 0.01\\
$l$ & Water loss due to evaporation  & 0.01\\
$q$ & Proportion of toxins in dead biomass  & 0.05\\
$r$ & Rate of water uptake  & 0.35\\
$s$ & Sensitivity of plants to toxicity $T$ & 0.2\\
$w$ & Washing out of toxins by precipitation  & 0.001\\
$D_B$ & Diffusion coefficient for Biomass $B$ & 0.01\\
$D_w$ & Diffusion coefficient  for water $W$  & 0.8\\
\hline
$p$ & Precipitation rate (bifurcation parameter)  & [0, 2]\\
 \hline
\end{tabular}
 \caption{Values of model parameters.}
 \label{tab:parameters}
\end{table}

\section{Linear Stability Analysis}
\label{sec:sec3anal}
In the following, we study the dynamics with respect to the precipitation rate parameter $p$. Initially, we seek for homogeneous solutions, setting the space and time derivatives in Eq.\ \eqref{eq:system} equal to zeros, thus obtaining the following nonlinear algebraic system: 
\begin{equation} 
\begin{split}
cB^2W-(d+sT)B & =0 \\
p-rB^2W-lW & =0\\
   q(d+sT)B-(k+wp)T & =0 ,
\end{split}
\label{eq:homzero1}
\end{equation}
The above system \eqref{eq:homzero1} has a trivial bare soil solution $(B_0,W_0,T_0)=(0,p/l,0)$.
 For a non-bare soil solution, i.e., when $B \neq 0$, we demonstrate the following proposition.
\begin{prop}
\label{prop:prop1}
Let the nonlinear algebraic system \eqref{eq:homzero1}. We define the functions $a_2(p)=sqcp+dr(k+wp)$, $a_1(p)=-(k+wp)cp$ and $a_0=(k+wp)dl$. Then, if the assumption 
\begin{equation}
    a_1^2-4a_0a_2>0 
    \label{eq:assum_prop_1}
\end{equation}
is satisfied, then the system \eqref{eq:system} has two non-bare soil branches of solutions. Furthermore, these two branches bifurcate and disappear when   
\begin{equation}
    a_1^2-4a_0a_2=0.
      \label{eq:assum_prop_1b}
\end{equation}
\end{prop}
\begin{proof}
We express the variables $W, T$ as a function of $B$ as
\begin{equation}
    W=\frac{p}{(rB^2+l)},
    \label{eq:water}
\end{equation}
 \begin{equation}
     T=\frac{1}{s} \left( \frac{cBp}{rB^2+l}-d \right)
     \label{eq:toxicity}
 \end{equation}
 Substituting the above in the third equation of the \eqref{eq:toxicity}, we obtain a second order equation with respect to the biomass $B$. 
\begin{equation} 
F(B,p)=a_2(p)B^2+a_1(p)B+a_0(p) =0 .
\label{eq:biomass}
\end{equation}
In case of a positive discriminant, i.e., for $\Delta = a_1^2-4a_0a_2>0 $, Eq.~\eqref{eq:biomass} has two solutions with respect to the parameter $p$. Specifically, Eq.~\eqref{eq:biomass} defines two branches of a parabola given by
\begin{equation} 
B_{1,2} =\frac{-a_1(p) \pm \sqrt{\Delta(p)}}{2a_2(p)}.
\label{eq:biomass_expil}
\end{equation}
The peak of the parabola results from $\Delta = a_1^2-4a_0a_2=0$.
\end{proof}
\begin{remark}
Substituting the values of the parameters from Table \ref{tab:parameters}, the assumption  Eq. \eqref{eq:assum_prop_1} is satisfied iff $p>p_{c_0}=0.64$, while the second assumption given by Eq.\eqref{eq:assum_prop_1b} is satisfied when $ p_{c_0}=0.64 \implies B=\frac{-a_1(0.64)}{2a_2(0.64)}=0.156$.
\end{remark}

\subsection{Stability  analysis of the homogeneous solution}
In this section we derive the stability conditions for the homogeneous solution, thus studying the existence of Turing bifurcations which mark the onset of dynamical instabilities \cite{Tur90}. Our system given by  Eq.\eqref{eq:toxicity} can be written in a compact form as: 
\begin{equation} 
\mathbf{u}_t=\mathbf{R(u)}+\mathbf{D}\mathbf{u}_{xx}, 
\label{eq:reac_diff}
\end{equation}
where, $\mathbf{u}=(B,W,T)$, $\mathbf{R(\mathbf{u})}=(f,g,h)$ with $f(\mathbf{u})=f(B,W,T)=cB^2W-(d+sT)B$, $g(\mathbf{u})=g(B,W,T)=p-rB^2W-lW $ and
$h(\mathbf{u})=h(B,W,T)= q(d+sT)B-(k+wp)T$. The constant matrix $\mathbf{D}$ is diagonal with its main diagonal containing the diffusion coefficients, i.e.,: 
\begin{equation} 
 \mathbf{D}=\begin{pmatrix}
D_B & 0 & 0\\
0 & D_W & 0\\
0 & 0 & 0\\
\end{pmatrix}
\label{eq:Diff_reac_diff}
\end{equation}
$D_{3,3}=0$ since the third equation of the system \eqref{eq:system} does not contain any diffusion term. Thus, we study the stability of a given homogeneous steady state solution $\mathbf{u_0}=(B_0,W_0,T_0)$. Towards this purpose, we introduce the perturbation $\mathbf{\delta u}= (\delta B,\delta W,\delta T)$ around the steady states, as $\mathbf{u}= \mathbf{u}_0+\mathbf{\delta u}$. Then, substituting the above into Eq.\eqref{eq:reac_diff} and using first order Taylor expansion for the reaction term $\mathbf{R(u)}$, we obtain the following linearized equation of Eq.\eqref{eq:reac_diff} around the steady state:
\begin{equation} 
\mathbf{(\delta u)}_t=\mathbf{D} \mathbf{(\delta u)}_{xx}+\mathbf{J(u_0)}\mathbf{\delta u}.
\label{eq:reac_diff_LInear}
\end{equation}
$\mathbf{J(u_0)}$ is the Jacobian matrix:
\begin{equation} 
 \mathbf{J(u_0)}= \left. \begin{pmatrix}
f_B & f_W & f_T\\
g_B & g_W & g_T\\
h_B & h_W & h_T\\ 
\end{pmatrix}
\right \vert_{u=u_0}.
\label{eq:jac_reac_diff}
\end{equation}
$\mathbf{\delta u}$ should satisfy the Neumann boundary condition \eqref{eq:NBC}, which implies that the $\mathbf{\delta u}$ has the following form 
\begin{equation} 
\mathbf{\delta u}={\mathbf{C} e^{\lambda t}\cos \frac{n\pi x}{L}}.
\label{eq:du_form}
\end{equation}
Then, the second order spatial derivatives (the Laplacian) read
\begin{equation}
    \mathbf{(\delta u)}_{xx}=-\left(\frac{n\pi}{L}\right)^2 \mathbf{\delta u},
\end{equation}
and the time derivative satisfies
  \begin{equation}
    \mathbf{(\delta u)}_{t}=\lambda \mathbf{\delta u}.
\end{equation}                            
Substituting the derivatives in Eq.\eqref{eq:reac_diff_LInear}, we obtain:
  \begin{equation}
   \lambda \mathbf{\delta u}= - \mathbf{D} \left(\frac{n\pi}{L}\right)^2 \mathbf{\delta u},+\mathbf{J(u_0)}\mathbf{\delta u},
\end{equation} 
or 
  \begin{equation}
   \left[ - \mathbf{D} \left(\frac{n\pi}{L}\right)^2+\mathbf{J(u_0)}- \lambda \mathbf{I}_{3}\right] \mathbf{\delta u}=\mathbf{0}.
   \label{eq:eigen}
\end{equation}                              
Eq. \eqref{eq:eigen} defines an eigenvalue-eigenfunction problem for the matrix $\mathbf{A}=- \mathbf{D} \left(\frac{n\pi}{L}\right)^2+\mathbf{J(u_0)}$, and for a nontrivial solution, the following condition must be satisfied 
  \begin{equation}
   \det[\mathbf{A}-\lambda \mathbf{I_3}]=\det \left[ - \mathbf{D} \left(\frac{n\pi}{L}\right)^2+\mathbf{J(u_0)}- \lambda \mathbf{I}_3\right] =\mathbf{0}.
   \label{eq:eigen_det}
\end{equation} 
In our case, the matrix $\mathbf{A}$ reads:
\begin{equation} 
 \mathbf{A}= \left. \begin{pmatrix}
f_B-D_B(\frac{n \pi}{L})^2& f_W & f_T\\
g_B & g_W-D_W(\frac{n \pi}{L})^2& g_T\\
h_B & h_W & h_T\\ 
\end{pmatrix}
\right \vert_{u=u_0},
\label{eq:A_matrix}
\end{equation}
and at the steady state ${u=u_0}$, we get:
\begin{equation} 
 \mathbf{A}=  \begin{pmatrix}
2cB_0W_0-(d+sT_0)-D_B(\frac{n \pi}{L})^2& cB_0^2 & -sB_0\\
-2rB_0W_0 & -rB_0^2-l-D_W(\frac{n \pi}{L})^2& 0\\
q(d+sT_0) & 0 & qsB_0-k-wp\\ 
\end{pmatrix}
\label{eq:A_matrix_subs}
\end{equation}
Eq.\ \eqref{eq:eigen_det} defines the characteristic equation of matrix $\mathbb{A}$ of third order: 
\begin{equation}
\label{eq:Char_Pol}
\begin{split}
   P(\lambda)& = \lambda^3+c_2 \lambda^2+c_1 \lambda+c_0\\
   & = \lambda^3-\Tr(A)\lambda^2-\frac{1}{2}(\Tr(A^2)-\Tr^2(A))\lambda-\det(A)=0,
\end{split}
\end{equation}
i.e., $c_2=-\Tr(A)$, $c_1=-\frac{1}{2}(\Tr(A^2)-\Tr^2(A))$ and $c_1=-\det(A)$. We state now a general criterion for the stability of the homogeneous solution.
\begin{prop} (Stability criterion)
The homogeneous steady state solution $\mathbf{u_0}=(B_0,W_0,T_0)$ of the reaction diffusion problem of Eq. \eqref{eq:system}, \eqref{eq:NBC} is stable if the following conditions hold:
\begin{equation}
  c_{2}>0,\quad c_{0}>0, \quad c_{2}c_{1}>c_{0}.
    \label{eq:assum_prop_2}
\end{equation}
\end{prop}
\begin{proof}
If for each $n\in \mathbb{N}$, the roots of Eq.\ \eqref{eq:Char_Pol} lie on the negative complex semi-plane, then the homogeneous solution is stable. Otherwise, if one exponent passes the imaginary axis, i.e., if $\Re(\lambda)>0$, the homogeneous solution loses stability and becomes unstable. Using the Routh–Hurwitz stability criterion \cite{Siet15,Son98} the homogeneous solution is stable if and only if $c_{2}>0, c_{0}>0$ and $ c_{2}a_{1}>a_{0}$. These conditions with the help of Eq.\ \eqref{eq:Char_Pol} can be written as:
\begin{equation}
\Tr(A)<0, \det(A)<0, -\frac{1}{2}(\Tr(A^2)-\Tr^2(A))*\Tr(A)>\det(A)
\label{eq:Conditions}
\end{equation}
\end{proof}

In the case of the trivial bare soil solution, i.e., for $(B_0,W_0,T_0)=(0,p/l,0)$, the proof of stability is trivial.
\begin{prop} 
The bare soil steady state solution $\mathbf{u_0}=(B_0,W_0,T_0)=(0,p/l,0)$
 of the reaction diffusion problem  \eqref{eq:system}, \eqref{eq:NBC} is always stable. 
\end{prop}
\begin{proof}
In this case, the matrix given in \eqref{eq:A_matrix_subs} takes the simple form:
\begin{equation} 
 \mathbf{A}= 
 \begin{pmatrix}
-d-D_B(\frac{n \pi}{L})^2& 0 & 0\\
0 & -l-D_W(\frac{n \pi}{L})^2& 0\\
qd & 0 & -k-wp\\ 
\end{pmatrix}.
\label{eq:A_matrix_subs_Bare}
\end{equation}
Thus, the eigenvalues of $\mathbf{A}$ are $\lambda_{1,n}=-d-D_B(\frac{n \pi}{L})^2<0$, $\lambda_{2,n} -l-D_W(\frac{n \pi}{L})^2<0$ and  $\lambda_{3,n}=-k-wp<0$ for each $n\in \mathbb{N}$ ($k, w, d, p >0$). Hence, the bare soil solution is always stable. 
\end{proof}

\subsection{Existence of Turing instability}
\label{sec:Tur_exis}
There are many different scenarios where the homogeneous solution loses stability. Since the characteristic polynomial is of third order, we can have one or two or even three real eigenvalues passing the imaginary axis. Another scenario is when two complex eigenvalues pass the imaginary axes. We state the following theorem.
\begin{prop} 
\label{prop:Tur}
The homogeneous steady state solution $\mathbf{u_0}=(B_0,W_0,T_0)$ with $B_0>0$ of the reaction diffusion problem given by Eq. \eqref{eq:system}, \eqref{eq:NBC} loses stability if
\begin{equation}
  c_{0}=-\det(A)=0
    \label{eq:Tur_1}
\end{equation}
\end{prop}
\begin{proof}
The simplest case of stability loss is when one leading eigenvalue passes the imaginary axis and becomes positive. When  $\lambda =0$, $P(0)=0$. Thus, directly from Eq. \eqref{eq:Char_Pol}, we obtain $c_0=-\det(A)=0$.
\end{proof}
We study the branch of positive biomass solutions, i.e., for $B>0$. The solution $(B_0,W_0,T_0),  B_0>0$ is given from Eq.\ \eqref{eq:water}-\eqref{eq:biomass_expil}. We  simplify the matrix $A$ given in \eqref{eq:A_matrix_subs}, using Eq.\ \eqref{eq:homzero1}. Dividing with $B$ the first equation in \eqref{eq:homzero1},  we obtain $cB_0W_0=d+sT_0$. From the second equation in \ \eqref{eq:homzero1}, we take $ rB_{0}^{2}+l=\frac{p}{W_0}$. Then, the $A$ is simplified to: 
\begin{equation} 
 \mathbf{A}=  \begin{pmatrix}
cB_0W_0-D_B(\frac{n \pi}{L})^2& cB_0^2 & -sB_0\\
-2rB_0W_0 & -\frac{p}{W_0}-D_W(\frac{n \pi}{L})^2& 0\\
qcW_0B_0 & 0 & qsB_0-k-wp\\ 
\end{pmatrix}.
\label{eq:A_matrix_subs2}
\end{equation}
\begin{remark}
In the case of the reaction diffusion problem \eqref{eq:system}, \eqref{eq:NBC}, the assumption in the proposition \ref{prop:Tur} reads:
\begin{equation} 
\begin{split}
 F(p,n,L)=\left( cB_0W_0-D_B (\frac{n \pi}{L})^2\right)\cdot  \left(-\frac{p}{W_0}-D_W(\frac{n \pi}{L})^2\right) \cdot (qsB_0-k-wp) \\ -2rcB_0^3W_0(-qsB_0+k+wp) +sB_0 \cdot\left(-\frac{p}{W_0}-D_W ( \frac{n \pi}{L})^2\right)\cdot qcW_0B_0=0 .
 \end{split}
\label{eq:detgeneq}
\end{equation}
\end{remark}
For constant $L$, Eq.\ \eqref{eq:detgeneq} defines implicitly the parameter $p$ as function of the physical number $n$. Solving Eq.\ \eqref{eq:detgeneq} for each value of $n, n=0,1,2,..$, we obtain the critical values of the parameter $p$. Fig.\ \ref{fig:wave_number}(a) shows the critical values $p_c=p_c(n)$ for $L=8$. For this size of the domain (specimen) only the modes for $n=1$ and $n=2$ result in the existence of a solution, while for $n=0$ and $n>2$ there are no critical values for $p_c$ ($p_c$ should also satisfy the conditions given by. \eqref{eq:assum_prop_1}, \eqref{eq:assum_prop_1b}, i.e., $p_c>0.64$. The first critical value comes for $n=2$ and the first critical precipitation rate is $p_{c_1}=1.14$. The second one comes for $n=1$ and $p_{c_2}=1.06$ (marked with filled circles in Fig.\ \ref{fig:wave_number}(a)).
\begin{figure}[t!]
\begin{center}
\hspace*{-.7cm} 
\begin{picture}(350,170)
\includegraphics[width=12.5cm]{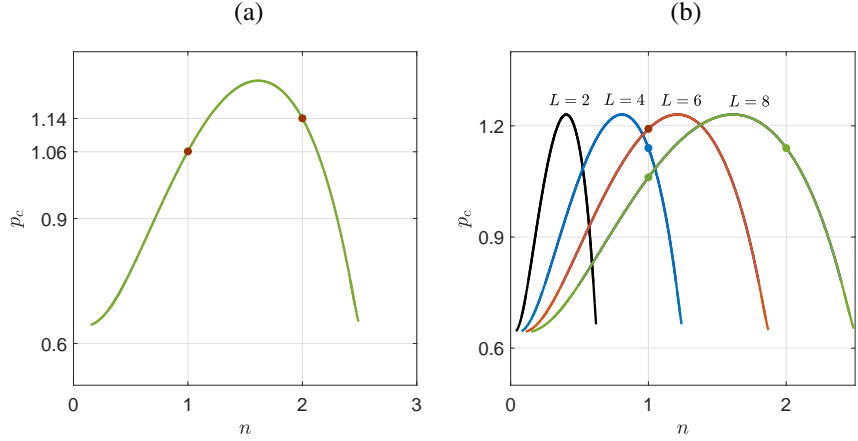}
\put(-260,160){(a)}
\put(-95,160){(b)}
\end{picture}
\end{center}
\caption{Critical values of the precipitation rate $p_c$ as a function of the natural number $n$, according to Eq.\ \eqref{eq:detgeneq}. \textbf{(a)} For $L=8$, the first critical value comes for $n=2$, which results for $p_{c_1}=1.14$. The second critical value results for $n=1$ and $p_{c_2}=1.06$ (both points marked with red filled circles). \textbf{(b)} The existence of critical values for the precipitation rate with respect to the domain size $L$. Here there are three scenarios: for higher values of $L$ (e.g. close to 8), there are two critical values of the precipitation rate $p_c$ (equivalent Turing modes of instability) for $n=1$ and $n=2$. As $L$ decreases, there is one critical value of $p_c$ (for $n=1$) and finally, when $L<L*=2.27$ there is no critical value of $p_c$ giving rise to Turing instability.}
\label{fig:wave_number}            
\end{figure}

\subsection{Size Effect on the Turing Instability}
The Turing eigenstability condition given by Eq.\ \eqref{eq:detgeneq} allow us to investigate the size effect on the multiplicity of the homogeneous solution (with $B>0$). For different values of $L$, we repeat the previous procedure,  for $n, n=0,1,2,..$, thus obtaining the corresponding critical values $p_{c_i}$. Fig.\ \ref{fig:wave_number}(b) shows the critical curves $p_c=p_{c,L}(n)$  for $L=2,4,6,8$. Higher values of $L$ increase the width of the curve, as it is depicted in Fig.\ \ref{fig:wave_number}(b), introducing modes of instability (or equivalent, new types of inhomogeneous solutions). For, $L=8$, there are two critical modes for $n=1$ and $n=2$, while for $L=6$ and $L=4$ there is only one mode of instability at $n=1$. Finally, for $L=2$ there is no instability mode.\par
We can identify the critical size $L=L^*$ where the modes of Turing instability disappear. Demanding $F(p=0.64,n=1,L)=0$ we obtain the critical value $L^*=2.27$. For values $L<L^*$ there are no Turing instabilities and  the upper branch change stability only at $p_{c_0}=0.64$ (see, proposition \ref{prop:prop1}). 

The impact of size $L$ on the system dynamics can be represented in the bifurcation diagram of homogeneous solutions. Fig.\ \ref{fig:bifanalytic} shows the bifurcations with respect to the precipitation parameter $p$, for two cases of the size $L$, one for $L=8>L^*$, Fig.\ \ref{fig:bifanalytic}(a), and one for $L<L^*$, Fig.\ \ref{fig:bifanalytic}(b). As we described in the case of $L=8$ the first critical parameter arises at  $p_{c_1}=1.14$) and then the upper branch looses its stability then, this branch of solutions remain unstable.
In the second case where $L<L^*$ the bifurcation curve is exactly the same, however there is a qualitative difference: since there is no Turing instability mode for $L<L^*$ the upper branch of Fig.\ \ref{fig:bifanalytic}(b) remains stable until the critical point of $p_{c_0}=0.64$ which bifurcates through a saddle node point.\par 
Another information that we gain from the linear analysis is the shape of the solution near the criticality (i.e., near the values $p_{c_2}, p_{c_1}$). The shape also depends on the size $L$. If the first instability arises for $n=2$ (e.g. as in the case of $L=8$), then the solution near the critical value will be $ \mathbf{x}= \mathbf{C}\cdot \cos(\frac{2\pi x}{L})$, with a spatial period $T=L$, which means that the profile is symmetric with respect to $L/2$. Instead, if the first instability appears at $n=1$ (which happens at low specimens $L$, e.g. for $L=6$ or $L=4$, see Fig.\ \ref{fig:wave_number}(b), then the solution (near the criticality) is $\mathbf{x}= \mathbf{C}\cdot \cos(\frac{\pi x}{L})$, with period $T=2L$. In this case, we have the half period profile, meaning that the shape of the solution will be skewed left or right half cosine. \par 
Thus, we conclude with a general rule that if the first mode of the instability results from an even physical number $n_0$ (i.e., mod$(n_0,2)=0$), then the profile, near the criticality, is symmetric with respect to $L/2$, in the interval $[0, L]$, while in the opposite case the profile is symmetric in the interval $[-L, L]$. 
\begin{figure}[t!]
\begin{center}
\hspace*{-1cm} 
\begin{picture}(350,150)
\includegraphics[width=13cm]{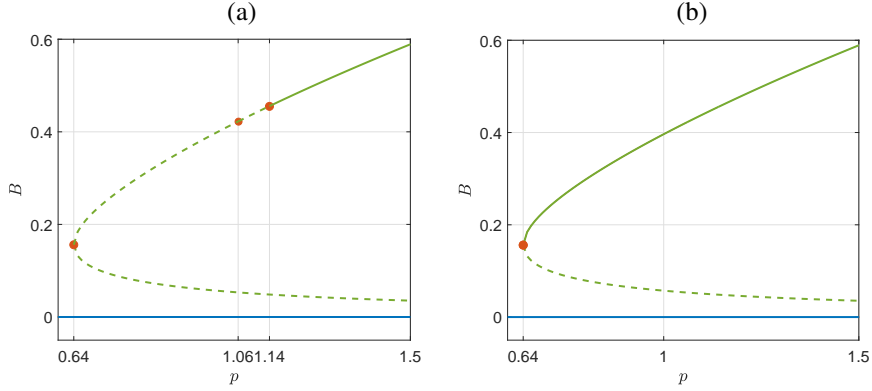}
\put(-270,140){(a)}
\put(-100,140){(b)}
\end{picture}
\end{center}
\caption{Bifurcation diagram of homogeneous solutions with respect to the precipitation rate $p$. Solid lines correspond to stable and dash to unstable state respectively. There are two sets of solutions. The first one is the bare soil branch ($B=0$) and the second set of homogeneous solutions with $B>0$ (as it is resulted from eq. \eqref{eq:biomass_expil}). The second set consists of 2 branches which are bifurcated at the critical value $p_{c_0}=0.64$. Depending on the size of domain $L$ the second set of solutions with $B>0$ shows different stability properties \textbf{(a)} For $L=8$, the upper branch loosing stability at  
(stable and unstable) On the stable branch of solutions with $B>0$ two critical values of $p$ are marked with red circles $p_{c_2}=1.06$ and $p_{c_1}=1.14$. These  values remark the onset of new inhomogeneous solutions, as we show in section \ref{sec:Tur_exis}}
\label{fig:bifanalytic}            
\end{figure}

\section{Symmetry properties of the vegetation dynamics model}
%
%
For every non-homogeneous solution $\boldsymbol{u}(x, t)$ of Eq. \eqref{eq:reac_diff}, there exists a solution $\boldsymbol{u}(x',t)$, in which $x'$ is obtained from x by the action of a symmetry group G defined as:
\begin{equation}
x'=\gamma x, \quad \forall \gamma \in G.
\end{equation}
Thus, the generic steady-state bifurcation from the homogeneous solution is always a pitchfork \cite{golubitsky2003symmetry}.\par 
In the above system, in the domain $[-L, \quad L]$,  the no-flux boundary conditions result to a $O(2)$ symmetry, thus being in a one-to-one correspondence with the domain $[0,\quad L]$.\par
It can be easily shown that Eq.\ref{eq:system} are invariant under the $Z_2$ reflection symmetry:
\begin{equation}
x \rightarrow L-x.
\end{equation}
As a consequence, the generic steady-state bifurcation emanating from the homogeneous solution is a pitchfork \cite{golubitsky2003symmetry,spiliotis2018analytical}.

\section{Numerical results} 
In this section, we first investigate the dynamics of the system \eqref{eq:system}, \eqref{eq:NBC} using numerical simulations. The previous analysis revealed the existence of critical values of 
the precipitation rate $p$, where the homogeneous solution loses stability due to Turing points. However, as discussed, the linear analysis, does not provide any information for the type-profile of the new solutions (especially far from the bifurcation point). Furthermore, in many cases new types of inhomogeneous solutions arise from secondary bifurcations points far from the homogeneous solutions (see e.g. in pp.120 in \cite{Nic77}) leading to complex (ecological) patterns, which linear analysis can not predict. Thus, numerical simulations may be used as a first step to discover the new types of solutions, and eventually multistability regions. However, as this way of analysis may discover the existence of only (some) stable solutions, in the next section, we complete the study by extracting all the branches of stable and unstable solutions by exploiting the arsenal of numerical bifurcation analysis.\par
The reaction diffusion model \eqref{eq:system}, \eqref{eq:NBC} is solved numerical using central finite differences in space, thus partitioning the domain $[0,L]$ with $L=8$ into $N$ equal intervals of size $h=\frac{L-0}{N}$. Considering the Neumann boundary conditions, we get a system of  $3 \cdot N-3$ ODEs, reading:
%
%
\begin{equation}
    \frac{d{\mathbf{u}}}{dt}=\mathbf{f}(\mathbf{u},p).
    \label{eq:syst_ode_1}
\end{equation}
The resulting dynamical system of ODEs is solved using the Matlab ode23s solver suitable for stiff problems. For our computations, we have used $N=40$, and the default ode option for the relative and absolute error (relative error $10^{-6}$ and absolute error $10^{-6}$). Larger values of $N$ resulted, for all practical purposes, quantitatively to same results.\par 
For large values of the precipitation rate $p>p_{c_1}=1.14$, the ecosystem exhibits two stable homogeneous stationary states, one corresponding to the homogeneous vegetated  state and the other corresponding to the bare soil solution. As the values of precipitation rate $p$ decreases, and in a perfect agreement  with the linear analysis, the homogeneous vegetated solution loses its stability (through  a Turing bifurcation at $p_{c_1}=1.14$, (see section \ref{sec:sec3anal}). As a consequence, depending on the initial conditions, the system may converge to one of two new types of bell-shaped and inverted bell-shaped symmetric but inhomogeneous solutions for the biomass $B$. These two solutions are reported in Fig.\ \ref{fig:bell1.1}(a,c), and they are obtained with initial conditions which are perturbations, in the center of the domain, in respect to the homogeneous solution: one positive (Fig. \ref{fig:bell1.1}(a)) and one negative perturbation, see Fig. \ref{fig:bell1.1}(c),  respectively.
\begin{figure}[t!]
\begin{center}
\begin{picture}(350,320)
\includegraphics[width=12cm]{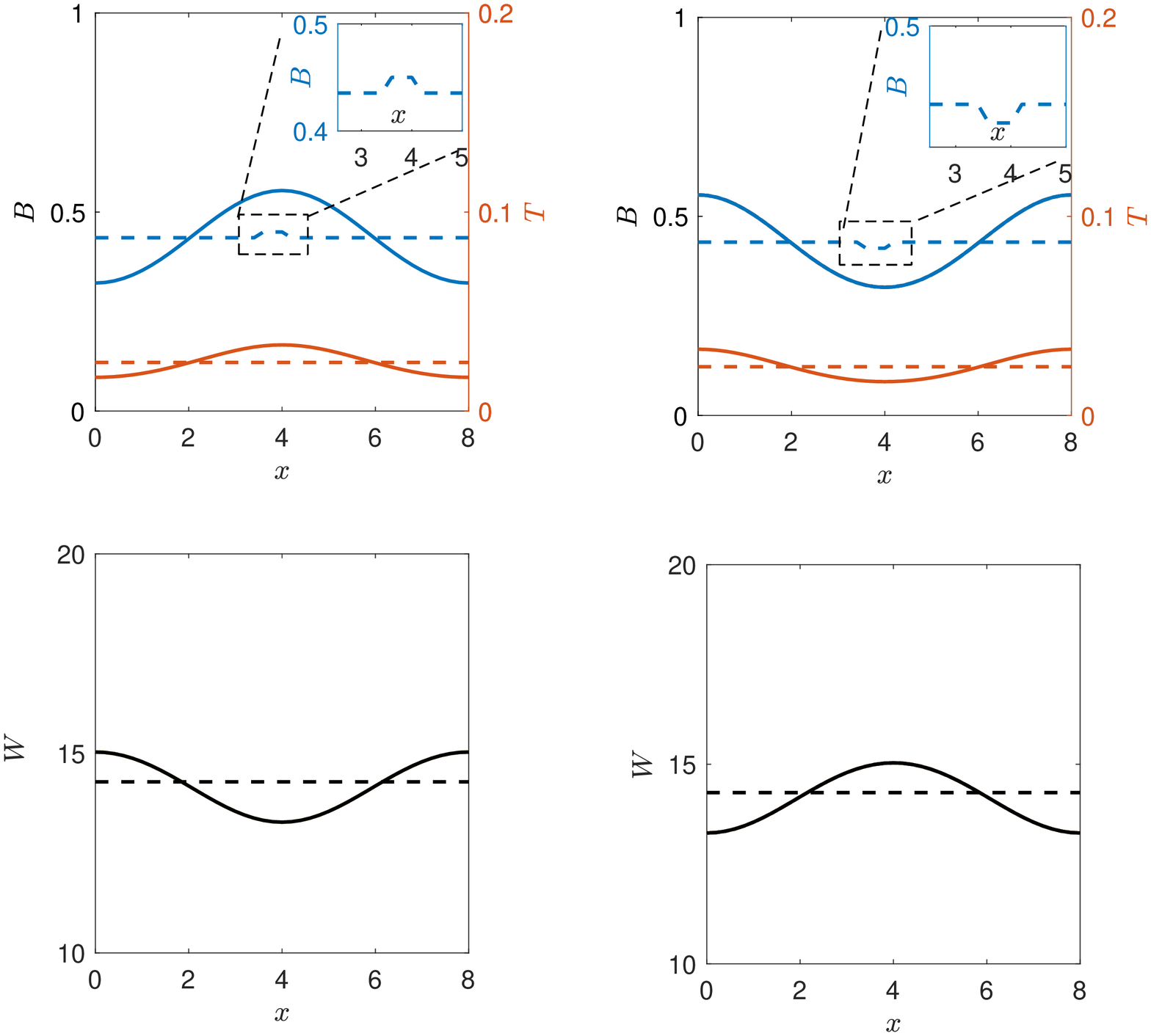}
\put(-255,310){(a)}
\put(-90,310){(c)}
\put(-255,152){(b)}
\put(-90,152){(d)}
\end{picture}
\end{center}
\caption{Evolution of the perturbed homogeneous solutions after the first bifurcation point of Fig.\ \ref{fig:bifanalytic}(a) at $p=1.1$. Dashed lines correspond to initial conditions, while solid lines depict the final steady state solution profile.\textbf{(a, b)} The biomass and toxicity profiles show a symmetric bell-shaped profile \textbf{(a)}, while the water shows an inverted bell-shaped profile  \textbf{(b)}. The homogeneous solutions are also given for comparison purposes. The inset depicts the results obtained by perturbing upwards the homogeneous solution at the center of the domain (i.e., we initialise as: $\textbf{x}_\text{{init}}=1.1 \cdot \textbf{x}_\text{{hom}}$, $x\in [3.8, 4.2]$). The initial value and the steady state solution (dashed-dot) of the toxicity is depicted in the right y-axis. \textbf{(c, d)} The biomass and toxicity profiles corresponding to inverted bell-shaped profiles, when the perturbation of the initial conditions is oriented down. The inset shows the perturbation of the homogeneous solution oriented down, i.e., $\textbf{x}_\text{{init}}=0.9 \cdot \textbf{x}_\text{{hom}}$, $x\in [3.7, 4.2]$). \textbf{(d)} The  water mass corresponding to a symmetric $\Lambda$ shape.}
\label{fig:bell1.1}            
\end{figure}
A further decrease of  the precipitation rate $p$ value, results to another critical transition around $p_{c_{3}}=0.99$. in particular, the bell-shaped solution looses the stability and two asymmetric conjugate inhomogeneous solutions appear. These new couple of solutions are shown in Fig. \ref{fig:asym.95} where the regime profiles are plotted for $p=0.95$, for different initial conditions.
We consider the initial conditions two perturbations of bell-shaped solution for biomass on left or right (zoom box in Fig. \ref{fig:asym.95}(a) and Fig. \ref{fig:asym.95}(c), respectively).
After a transient time, the system converges to two different regime stable solutions reported in  Fig. \ref{fig:asym.95}(a),(c).
%
We comment here that although the bell-shaped solution disappears after the critical point around $p_{c_{3}}=0.99$, the inverted bell-shaped solution remains stable  until the value $p_{c_{4}}=0.91$ (see Fig.\ \ref{fig:downbell0.95}). For lower values of $p$ (i.e., $p<p_{c_{4}}$) the inverted bell-shaped solution is vanished and the system exhibits only skewed left or skewed right solutions.
%
 
\begin{figure}[t!]
\begin{center}
\begin{picture}(350,320)
\includegraphics[width=12cm]{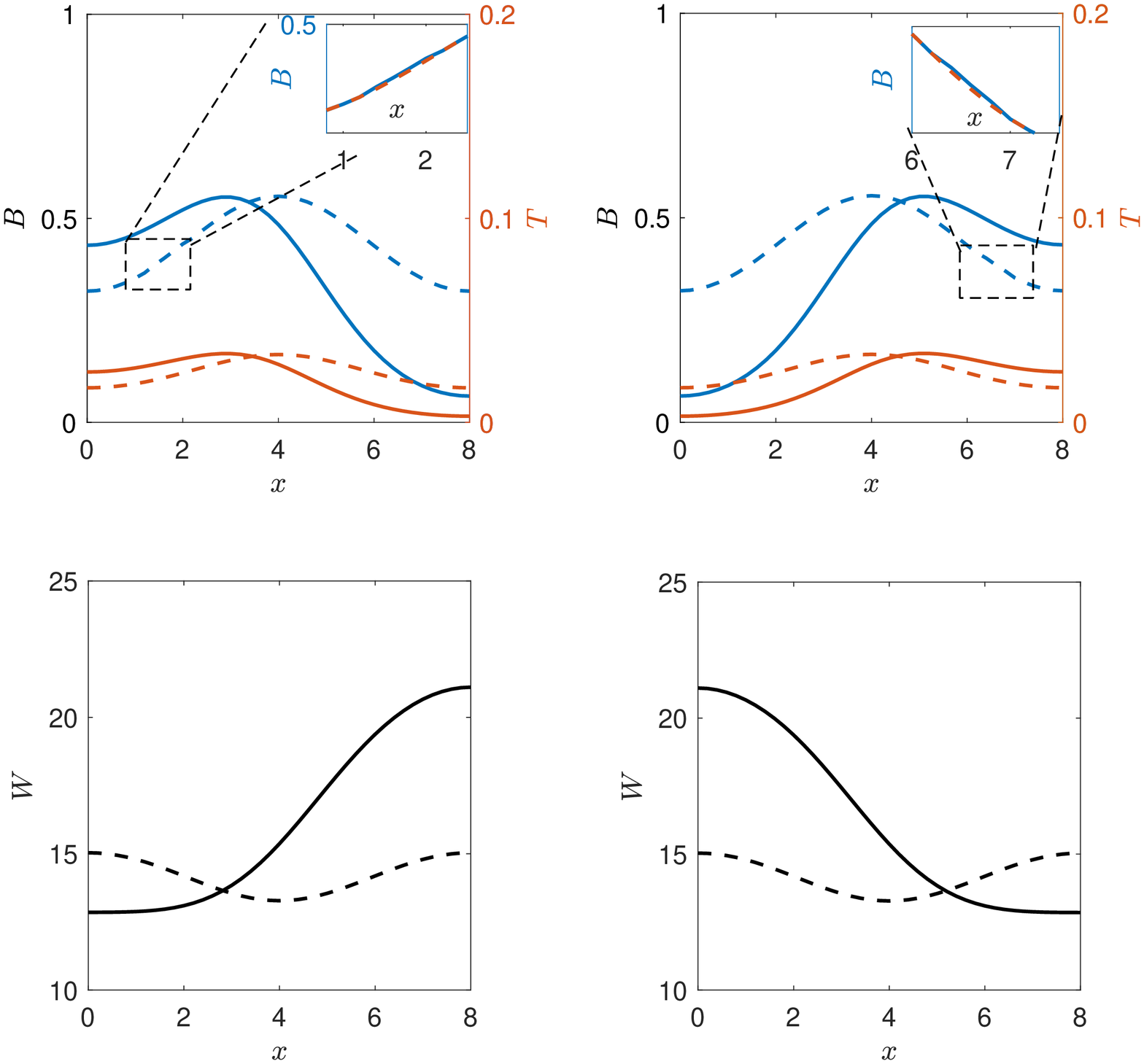}
\put(-255,303){(a)}
\put(-90,303){(c)}
\put(-255,153){(b)}
\put(-90,150){(d)}
\end{picture}
\end{center}
\caption{Profile of solutions close to a critical value where the system dynamics deviates from the bell-shaped solution of Fig.\ \ref{fig:bell1.1} i.e., simulations for  $p=0.95$. Dashed lines correspond to the initial conditions, while solid lines depict the final steady states solutions. \textbf{(a)} Biomass, starting from initial condition close to solution of Fig. \ref{fig:bell1.1}(a); after a transient period, the system converges to a skewed left inhomogeneous solution. The inset shows the initial condition, which is the bell-shaped of Fig. \ref{fig:bell1.1}(a) (as red dash line) perturbed on the left side of the domain (i.e., $\textbf{x}_\text{{init}}=1.1 \cdot \textbf{x}_\text{{Bell}}$ for $x\in [1.1, 2.1]$).\textbf{(b)} Water mass dynamics exhibits a right asymmetric behavior. 
\textbf{(c)} A biomass skewed right inhomogeneous solution appears when the perturbation of the initial conditions is oriented right-up.\textbf{(d)} Water mass dynamics exhibits an opposite left asymmetric behavior. }
\label{fig:asym.95}            
\end{figure}

\begin{figure}[t!]
\begin{center}
\begin{picture}(350,190)
\includegraphics[width=12cm]{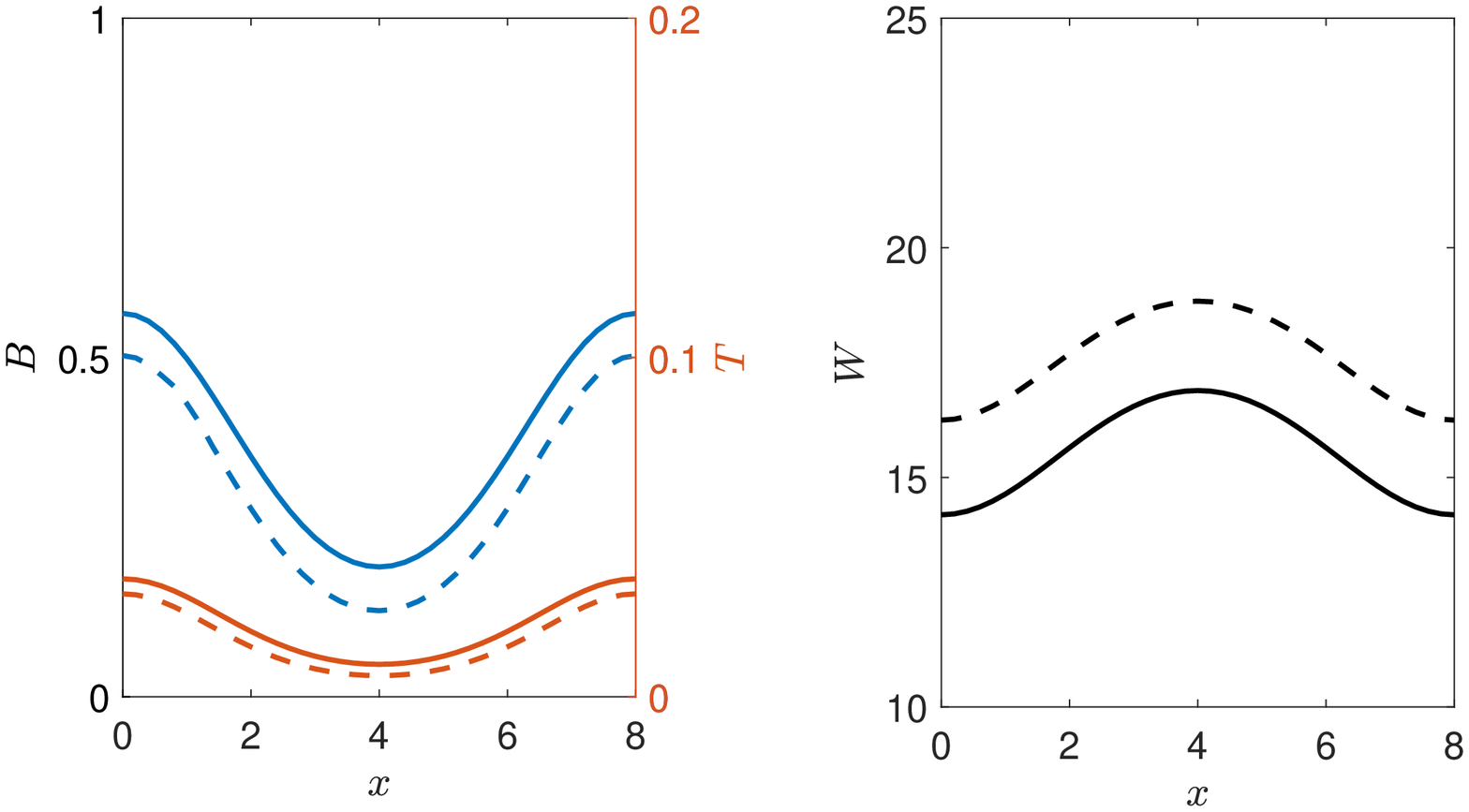}
\put(-250,180){(a)}
\put(-90,180){(b)}
\end{picture}
\end{center}
\caption{Persistence of the inverted bell-shaped solution. Temporal simulation for $p=0.95$. Dashed lines correspond to the initial conditions, while solid lines depict the final steady state solutions. In contrast to the bell-shaped solution (Fig. \ref{fig:asym.95}(a),(b) where the bell-shapes solution loses stability at $p=0.99$), the inverted bell-shaped solution keeps its stability until $p=0.91$, \textbf{(a)} for biomass and toxicity and \textbf{(b)} for water.}
\label{fig:downbell0.95}            
\end{figure}
This kind of solutions exist for even smaller values of $p$ and finally at some critical point (around $p=0.44$) the system depicts only the bare soil homogeneous solution ($B=0$ and $T=0$) which is permanent as $p \rightarrow 0$.

\section{Numerical Bifurcation Analysis}

In order to systematically discover and trace both stable and unstable branches of steady state solutions, that are unreachable using the linear analysis of section \ref{sec:sec3anal} or with numerical temporal simulations resented in the previous section, and to accurately estimate the location of the critical points which mark the onset of phase transitions we resorted to the arsenal of numerical bifurcation theory. 
For the transformed system of Eq.\ \eqref{eq:syst_ode_1} the steady states are computed as solution of equation:
   \begin{equation}
   \mathbf{f}(\mathbf{u},p)=0,
   \label{eq:con_f}
\end{equation}
 The numerical bifurcation analysis is implemented with the aid of MatCont \cite{Matcont,matcont2}.
The Matcont algorithm is based on a predictor-corrector method \cite{Matcont,matcont2}. Suppose that we have detected a point 
$\mathbf{x}_i=(\mathbf{u}_{i},p_i)$ along the curve which is defined from eq.\ \eqref{eq:con_f}, also let $\mathbf{v}_{i}$ a normalized tangent vector  at $\mathbf{x}_{i}$ , i.e. $\mathbf{f}_x(\mathbf{x_i}) \cdot \mathbf{v}_i=0$, and $||\mathbf{v}_{i}||=1$. The computation of the point $\mathbf{x}_{i+1}$ is made in two steps, first using a predictor (predicting a new point) and then, correcting the new point using Newton iterations.

As a predictor $\widetilde{\mathbf{x}}_{i+1}$, we used a point on the tangent direction, i.e.:
 \begin{equation}
 \widetilde{\mathbf{x}}_{i+1}=\mathbf{x}_{i}+h \mathbf{v}_{i},
,\end{equation}
where $h$ is a small-selected step. The correction uses an augmented with one equation Newton scheme. We add  the equation 
\begin{equation}
   g(\mathbf{x})=(\mathbf{x}-\widetilde{\mathbf{x}}_{i+1})\cdot \mathbf{v}_i=0,
   \label{eq:pseodo_arc}
\end{equation}
which is the well-known pseudo-arc-length continuation scheme, according to which, the final point results as the intersection of the  hyperplane passing through $\widetilde{\mathbf{x}}_{i+1})$ and the tangent predictor, i.e.,:
\begin{equation}
   \mathbf{x}^{k+1}= \mathbf{x}^{k}-\mathbf{F}_{x}^{-1}(\mathbf{x}^k)\mathbf{F}(\mathbf{x}^k),
   \label{eq:Newton_pseodo_arc}
\end{equation}
with  $\mathbf{F}=(\mathbf{f},g)^{T}$ and $\mathbf{F}_x$ is the Jacobian matrix of $\mathbf{F}$.
The Newton-Raphson iterations termination criteria are the function and the step tolerance with tolerances set less than a specific value (here at $10^-6$) $||\mathbf{F}(\mathbf{x}^{k+1})|| < \epsilon_1$ and an additional accuracy condition $|| \delta \mathbf{x} || < \epsilon_2$, where $\delta \mathbf{x}$ is the last Newton-Raphson correction.
\begin{figure}[t]
\begin{center}
\hspace*{-4.6cm}
\begin{picture}(350,250)
\includegraphics[width=17cm]{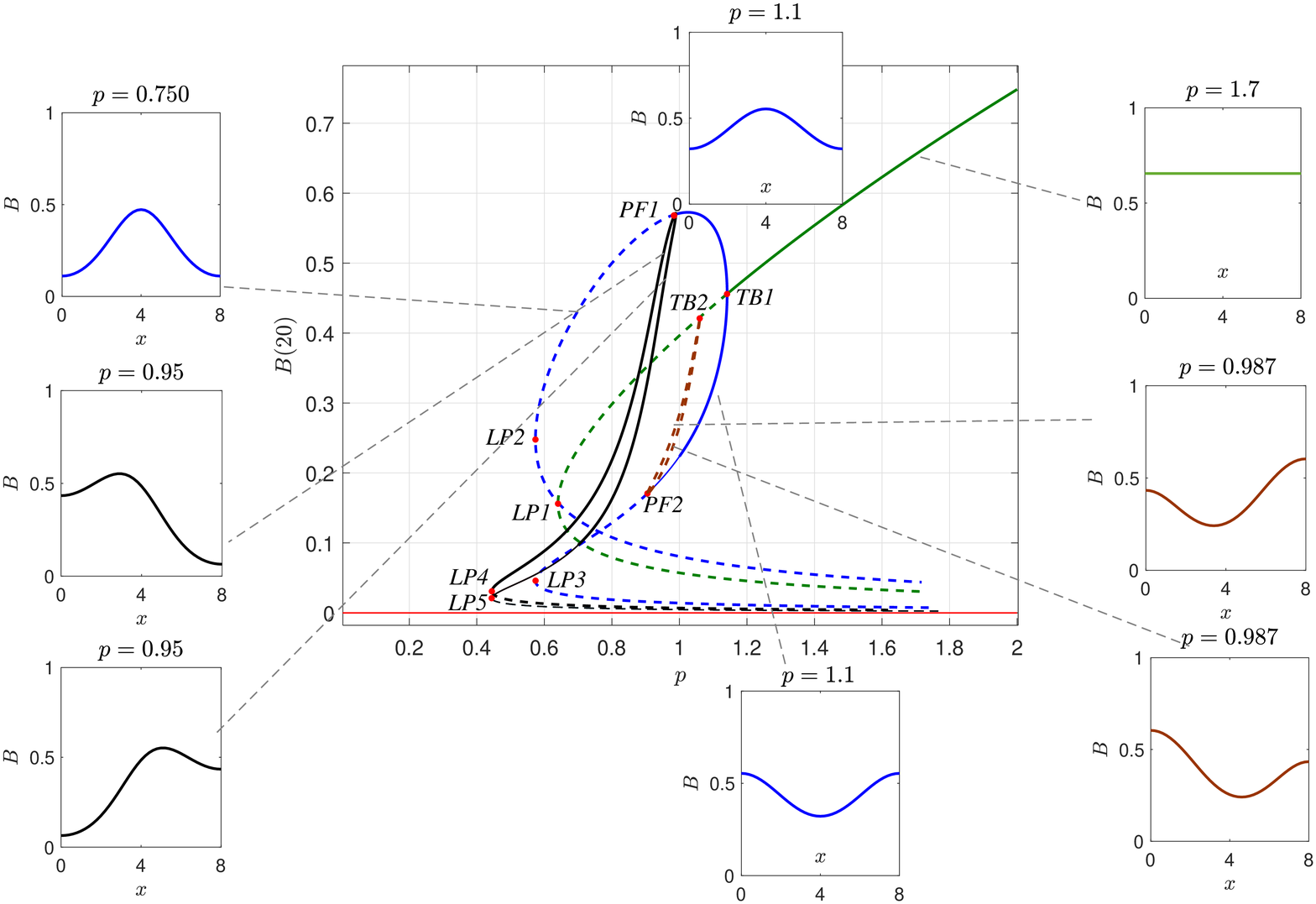}
\end{picture}
\end{center}
\caption{Bifurcations diagram with respect to the precipitation rate $p$. Solid lines correspond to stable states and dashed lines to unstable states. The system manifests rich nonlinear dynamics, with symmetric and asymmetric solutions, multistability and symmetry breaking bifurcations. The gray dashed lines depict the corresponding  biomass spatial profiles of the steady state solutions. There are five branches: green and red which correspond the homogeneous profiles, the blue branch which stands for the bell and inverted bell-shaped solutions, the black which is the skewed and left and right profile and deep red unstable branch with asymmetric  (almost inverted bell-shaped) solutions. The description and the exact positions of bifurcation points highlighted with labels are given in Table \ref{tab:tip_points} }
\label{fig:Biftot}            
\end{figure}
\begin{table}[t]
\centering
\begin{tabular}{ p{1.8cm}||p{6cm}|p{2cm} }
 \hline
 \hline
 symbol in Figure \ref{fig:Biftot} & Description of symbols appears in fig. & critical value \\
 \hline \hline 
 $LP1$ & saddle node bifurcation of homogeneous solution & $p_{c_0}=0.64$\\
$TB1$ & first Turing instability of homogeneous solution with $B>0$ & $p_{c_1}=1.14$ \\
$TB2$ & second Turing instability of homogeneous solution  with $B>0$ & $p_{c_2}=1.06$\\
$PF1$ & pitchfork bifurcation of the bell-shaped profile & $p_{c_3}=0.99$\\
$PF2$ &  pitchfork bifurcation of the inverted bell-shaped profile & $p_{c_4}=0.91$\\
$LP2$ & bell-shaped saddle node bifurcation &$p_{c_5}=0.54$\\
$LP3$ & inverted bell-shaped saddle node bifurcation  & $p_{c_5}=0.54$\\
$LP4$ & skewed left asymmetric profile saddle node bifurcation  & $p_{c_6}=0.44$ \\
$LP5$ & skewed right asymmetric profile saddle node bifurcation  & $p_{c_6}=0.44$ \\
\hline
\end{tabular}
 \caption{Critical values of bifurcation points (or tipping points) as they appear in Fig.\ \ref{fig:Biftot}}
 \label{tab:tip_points}
\end{table}
\subsection{The Bifurcation Diagram} 

Fig.\ \ref{fig:Biftot} depicts the resulting bifurcation diagram. Characteristic  profiles of the solutions along the branches are also shown as insets. Starting from $p=2$ and going downhill, the system shows only two branches of stable homogeneous solutions, one with $B>0$ and the second branch with $B=0$.
 At a critical point  $p_{c_1}=1.14$ (marked as $TB1$, in Fig.\ \ref{fig:Biftot}), corresponding to a Turing bifurcation, the homogeneous solution (with $B>0$) loses its stability and gives birth to two new inhomogeneous solutions of a bell-shaped and inverted bell-shaped (see insets and Fig.\ \ref{fig:bell1.1}). %
 Decreasing more the value of $p$, the branch of homogeneous solutions remains unstable and on this branch, at the point $p_{c_2}=1.06$, a second point of Turing instability appears (marked as $TB2$, in Fig.\ \ref{fig:Biftot}).
At this second Turing bifurcation point ($TB2$), two new unstable branches of non-homogeneous solutions appear. Finally, the homogeneous unstable branch bifurcates through a saddle node bifurcation at $p_{c_0}=0.64$  (marked as $LP1$, in Fig.\ \ref{fig:Biftot}).\par 
The bell-shaped and inverted bell-shaped solutions, which emerge from, $TB1$ are stable. The upper branch (with the bell-shaped patterns) remains stable until the  critical point $p_{c_3}=0.99$ (marked as $PF1$ in Fig.\ \ref{fig:Biftot}).  Then, the solution on this branch loses its stability and bifurcates with two new branches of inhomogeneous solutions, which are symmetrically conjugated. This type of secondary bifurcation can not be predicted from the linear analysis of the homogeneous solution. Remarkable, the inverted bell-shaped patterning keeps stability until $p_{c_4}=0.91$  (marked as $PF2$ in Fig.\ \ref{fig:Biftot}). Thus, the unstable branches emerging from $TB2$ connect $PF2$ and $TB2$ points. Finally, both unstable branches of bell-shaped and inverted bell-shaped patterns experience a saddle node bifurcation at the critical value $p_{c_5}=0.54$ (marked as $LP2$ and $LP3$ in Fig.\ \ref{fig:Biftot}).\par
Furthermore, at the  point $PF1$, two new stable branches of inhomogeneous solutions arise. The profiles are skewed left and right solutions (see also Fig.\ \ref{fig:asym.95}). These branches lose stability under a saddle-node bifurcation, which takes place at the critical value $p_{c_6}=0.44$ marked as $LP4, LP5$ in Fig.\ \ref{fig:Biftot}. The profile of solutions is depicted with black color in the insets of Fig. \ref{fig:Biftot}.\par   
Concluding, the system reveals a rich nonlinear dynamical behavior characterized by symmetry and symmetry breaking bifurcations and coexistence of multiple stable and unstable regimes. For $p \rightarrow 0$ the system exhibits only stable bare-soil solutions. Multistability is observed from $LP4$ to $PF2$ with three stable solutions (the bare-soil and two symmetrically conjugate solutions (depicted with black color lines in the insets  of Fig. \ref{fig:Biftot}. Whereas from $PF2$ to $PF1$ the system provides four stable regimes (bare soil, two inhomogeneous symmetrically conjugate solutions and the inverted bell-shaped solution). From $PF1$ to $TB1$ there are three stable solutions (inverted bell-shaped, bell-shaped and bare soil solutions) and finally after $TB1$ we have the two homogeneous solutions, corresponding to the vegetated and soil solutions 

In Fig.\ref{fig:symetBell}, we illustrate the symmetry breaking-symmetry  of the solutions. A a consequence of the Turing bifurcation at the point $p_{c_1}=1.14$ (marked as TB1), there is a symmetry breaking of the homogeneous solution and two new solutions appear (bell-shaped and inverted bell-shaped profiles in Fig.\ref{fig:symetBell}(a)). These solutions near the TB1 exhibit a symmetry, that one is the reflection of the other around the homogeneous solution. However, this symmetry is not preserved far from the TB1 (where the linearization is not valid and nonlinearity becomes significant), as it is shown in Fig.\ref{fig:symetBell}(b). Furthermore,in Fig.\ref{fig:symetBell}(c-d) are shown solutions arised from $PF1$ the  symmetric between them, reflecting a conjugate symmetric pattern.            
\section{Discussion}
We performed bifurcation analysis of a biomass-water-toxicity model with respect to the precipitation. The model consists of a set of two PDEs and one ODE that describe qualitatively the pattern formation in semi-arid zones as the precipitation decreases before the occurrence of desertification.  We first performed a linear stability analysis for the solution branch of the homogeneous state to provide analytically: (a) the conditions for the appearance of Turing bifurcations that mark the onset of pattern formation, and, (b) the dependence of the Turing bifurcations on the size of the domain. From these critical points, arise two inhomogeneous solution branches, which are symmetrical to the axis of the homogeneous solution. This is a known symmetry-breaking phenomenon, due to the Turing bifurcation, which with the zero flux boundary conditions has the characteristic of a pitchfork bifurcation \cite{dillon1994pattern,krause2021modern,woolley2022boundary}. 

Here, we argue, based on numerical evidence, that the Turing-type symmetry breaking is fundamentally different from the symmetry-breaking bifurcations encountered in dynamical systems with $R^2$ symmetry. In particular, the numerical bifurcation analysis, reveals also pitchfork bifurcations which break the reflection symmetry induced by the boundary conditions. Differently from the patterns arising from the Turing-Pitchfork-type bifurcation arising from  zero-flux boundary conditions, here the reflection-conjugate patterns experience the same bifurcations and stability and they always show-up in pair. As discussed also in Krause et al. \cite{krause2021modern}, while the linear stability analysis is formally valid around the Turing bifurcation from the homogeneous solution, it does not provide any information about possible subsequent bifurcations away from the uniform-equilibrium solution. In fact, we show that after the initial Turing symmetric instability, a secondary bifurcation arises which splits the solution branches in two distinct, unstable, asymmetric steady states, followed by a reverse asymmetric Turing bifurcation, in which the
asymmetric equilibrium branches gains again stability. A similar mechanism has been observed in a two-layer model consisting of a pair of coupled reaction-diffusion equations \cite{yang2004symmetric}. Regarding the vegetation pattern formation, such asymmetric  patterns have been observed in response to localized differences in soil-water availability \cite{tarnita2017theoretical}. 
\begin{figure}[t!]
\begin{center}
\hspace*{-.7cm} 
\begin{picture}(350,330)
\includegraphics[width=14cm]{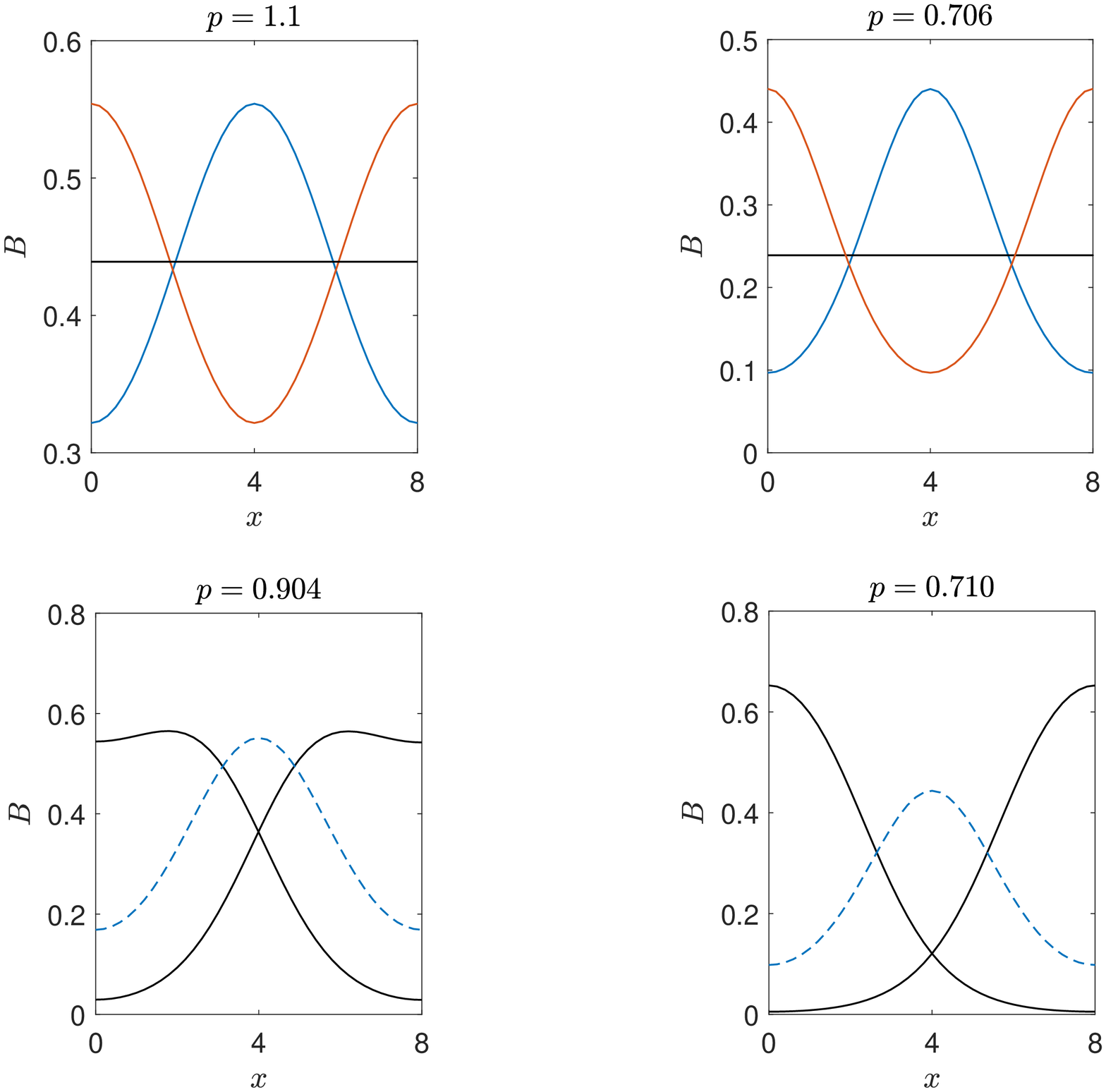}
\put(-310,320){(a)}
\put(-120,320){(b)}
\put(-310,150){(c)}
\put(-120,150){(d)}
\end{picture}
\end{center}
\caption{Symmetry and symmetry-breaking solutions. \textbf{(a)} Near the critical point, $TB1$ both bell-shaped and inverted bell-shaped solutions are symmetric with respect to the homogeneous solution (constant horizontal line with black colour). \textbf{(b)} Far from the $TB1$, the profiles are not any more symmetric. \textbf{(c,d)} Symmetry preserved along the branches of skewed-left and skewed-right inhomogeneous solution (black branches of Fig. \ref{fig:Biftot}). After the point, $PF1$, the two branches (black colour) preserve their symmetry.}
\label{fig:symetBell}            
\end{figure}
\clearpage
\section*{Declarations}
\subsection*{Conflicts of interest/Competing interests} The authors have no conflict of interests to disclose
\subsection*{Availability of data and material (data transparency)} Not applicable

\bibliographystyle{unsrt}

\bibliography{references}  

\begin{thebibliography}{10}

\bibitem{vincenot2016spatial}
C.~E. Vincenot, F.~Carteni, S.~Mazzoleni, M.~Rietkerk, and F.~Giannino.
\newblock Spatial self-organization of vegetation subject to climatic
  stress—insights from a system dynamics—individual-based hybrid model.
\newblock {\em Frontiers in plant science}, 7:636, 2016.

\bibitem{zhao2019shaping}
L.~X. Zhao, C.~Xu, Z.~M. Ge, J.~Van De~Koppel, and Q.~X. Liu.
\newblock The shaping role of self-organization: linking vegetation patterning,
  plant traits and ecosystem functioning.
\newblock {\em Proceedings of the Royal Society B}, 286(1900):20182859, 2019.

\bibitem{callaway2021belowground}
R.~M. Callaway, E.~Meron, et~al.
\newblock Belowground feedbacks as drivers of spatial self-organization and
  community assembly.
\newblock {\em Physics of Life Reviews}, 38:1--24, 2021.

\bibitem{kefi2007spatial}
S.~K{\'e}fi, M.~Rietkerk, C.~L. Alados, Y.~Pueyo, V.~P. Papanastasis,
  A.~ElAich, and P.~C. De~Ruiter.
\newblock Spatial vegetation patterns and imminent desertification in
  mediterranean arid ecosystems.
\newblock {\em Nature}, 449(7159):213--217, 2007.

\bibitem{tarnita2017theoretical}
C.~E Tarnita, J.~A. Bonachela, E.~Sheffer, J.~A. Guyton, T.~C. Coverdale, R.~A.
  Long, and R.~M. Pringle.
\newblock A theoretical foundation for multi-scale regular vegetation patterns.
\newblock {\em Nature}, 541(7637):398--401, 2017.

\bibitem{bonanomi2014ring}
G.~Bonanomi, G.~Incerti, A.~Stinca, F.~Carten{\'i}, F.~Giannino, and
  S~Mazzoleni.
\newblock Ring formation in clonal plants.
\newblock {\em Community Ecology}, 15(1):77--86, 2014.

\bibitem{silvertown1992cellular}
J.~Silvertown, S.~Holtier, J.~Johnson, and P.~Dale.
\newblock Cellular automaton models of interspecific competition for space--the
  effect of pattern on process.
\newblock {\em Journal of Ecology}, pages 527--533, 1992.

\bibitem{pascual2002cluster}
M.~Pascual, M.~Roy, F.~Guichard, and G.~Flierl.
\newblock Cluster size distributions: signatures of self--organization in
  spatial ecologies.
\newblock {\em Philosophical Transactions of the Royal Society of London.
  Series B: Biological Sciences}, 357(1421):657--666, 2002.

\bibitem{vincenot2017plant}
C.~E. Vincenot, F.~Carten{\'i}, G.~Bonanomi, S.~Mazzoleni, and F.~Giannino.
\newblock Plant--soil negative feedback explains vegetation dynamics and
  patterns at multiple scales.
\newblock {\em Oikos}, 126(9):1319--1328, 2017.

\bibitem{bonanomi2005negative}
G.~Bonanomi, F.~Giannino, and S.~Mazzoleni.
\newblock Negative plant--soil feedback and species coexistence.
\newblock {\em Oikos}, 111(2):311--321, 2005.

\bibitem{carteni2012negative}
F.~Carten{\'i}, A.~Marasco, G.~Bonanomi, S.~Mazzoleni, M.~Rietkerk, and
  F.~Giannino.
\newblock Negative plant soil feedback explaining ring formation in clonal
  plants.
\newblock {\em Journal of theoretical biology}, 313:153--161, 2012.

\bibitem{marasco2014vegetation}
A.~Marasco, A.~Iuorio, F.~Carten{\'i}, G.~Bonanomi, D.~M. Tartakovsky,
  S.~Mazzoleni, and F.~Giannino.
\newblock Vegetation pattern formation due to interactions between water
  availability and toxicity in plant--soil feedback.
\newblock {\em Bulletin of mathematical biology}, 76(11):2866--2883, 2014.

\bibitem{severino2017effects}
G.~Severino, F.~Giannino, F.~Carten{\'i}, S.~Mazzoleni, and D.~M. Tartakovsky.
\newblock Effects of hydraulic soil properties on vegetation pattern formation
  in sloping landscapes.
\newblock {\em Bulletin of Mathematical Biology}, 79(12):2773--2784, 2017.

\bibitem{rietkerk2008regular}
M.~Rietkerk and J.~Van~de Koppel.
\newblock Regular pattern formation in real ecosystems.
\newblock {\em Trends in ecology \& evolution}, 23(3):169--175, 2008.

\bibitem{turing1952}
A.~M. Turing.
\newblock The chemical basis of morphogenesis.
\newblock {\em Philosophical Transactions of the Royal Society of London,
  Series B}, 237:37--72, 1952.

\bibitem{Tur90}
A.~M. Turing.
\newblock The chemical basis of morphogenesis.
\newblock {\em Bulletin of Mathematical Biology}, 52(1-2):153--197, 1990.

\bibitem{maini1997spatial}
P.~Maini, K.~J. Painter, and H.~N.~P. Chau.
\newblock Spatial pattern formation in chemical and biological systems.
\newblock {\em Journal of the Chemical Society, Faraday Transactions},
  93(20):3601--3610, 1997.

\bibitem{ball2015forging}
P.~Ball.
\newblock Forging patterns and making waves from biology to geology: a
  commentary on turing (1952)‘the chemical basis of morphogenesis’.
\newblock {\em Philosophical Transactions of the Royal Society B: Biological
  Sciences}, 370(1666):20140218, 2015.

\bibitem{krause2021modern}
A.~L. Krause, E.~A. Gaffney, P.~K. Maini, and V.~Klika.
\newblock Modern perspectives on near-equilibrium analysis of turing systems.
\newblock {\em Philosophical Transactions of the Royal Society A},
  379(2213):20200268, 2021.

\bibitem{klausmeier1999regular}
C.~A. Klausmeier.
\newblock Regular and irregular patterns in semiarid vegetation.
\newblock {\em Science}, 284(5421):1826--1828, 1999.

\bibitem{gowda2016assessing}
K.~Gowda, Y.~Chen, S.~Iams, and M.~Silber.
\newblock Assessing the robustness of spatial pattern sequences in a dryland
  vegetation model.
\newblock {\em Proceedings of the Royal Society A: Mathematical, Physical and
  Engineering Sciences}, 472(2187):20150893, 2016.

\bibitem{lamb2006hopf}
J.~S.~W. Lamb, I.~Melbourne, and C.~Wulff.
\newblock Hopf bifurcation from relative periodic solutions; secondary
  bifurcations from meandering spirals.
\newblock {\em Journal of Difference Equations and Applications},
  12(11):1127--1145, 2006.

\bibitem{spiliotis2018analytical}
K.~G/ Spiliotis, L.~Russo, C.~Siettos, and E.~C. Aifantis.
\newblock Analytical and numerical bifurcation analysis of dislocation pattern
  formation of the walgraef--aifantis model.
\newblock {\em International Journal of Non-Linear Mechanics}, 102:41--52,
  2018.

\bibitem{aragon2012nonlinear}
J.~L. Arag{\'o}n, R.~A. Barrio, T.~E. Woolley, R.~E. Baker, and P~.~K. Maini.
\newblock Nonlinear effects on turing patterns: Time oscillations and chaos.
\newblock {\em Physical Review E}, 86(2):026201, 2012.

\bibitem{banerjee2012turing}
M.~Banerjee and S.~Banerjee.
\newblock Turing instabilities and spatio-temporal chaos in ratio-dependent
  holling-tanner model.
\newblock {\em Mathematical biosciences}, 236(1):64--76, 2012.

\bibitem{barrio2002size}
R.~A. Barrio, P.~K. Maini, J.~L. Arag{\'o}n, and M.~Torres.
\newblock Size-dependent symmetry breaking in models for morphogenesis.
\newblock {\em Physica D: Nonlinear Phenomena}, 168:61--72, 2002.

\bibitem{russo2019bautin}
L.~Russo, K.~Spiliotis, F.~Giannino, S.~Mazzoleni, and C.~Siettos.
\newblock Bautin bifurcations in a forest-grassland ecosystem with
  human-environment interactions.
\newblock {\em Scientific reports}, 9(1):1--8, 2019.

\bibitem{spiliotis2021analytical}
K.~Spiliotis, L.~Russo, F.~Giannino, and C.~Siettos.
\newblock Analytical and numerical bifurcation analysis of a forest ecosystem
  model with human interaction.
\newblock {\em ESAIM: Mathematical Modelling and Numerical Analysis},
  55:S653--S675, 2021.

\bibitem{satnoianu2000turing}
R.~A. Satnoianu, M.~Menzinger, and P.~K. Maini.
\newblock Turing instabilities in general systems.
\newblock {\em Journal of mathematical biology}, 41(6):493--512, 2000.

\bibitem{henderson2016alternative}
K.~A. Henderson, C.~T. Bauch, and M.~Anand.
\newblock Alternative stable states and the sustainability of forests,
  grasslands, and agriculture.
\newblock {\em Proceedings of the National Academy of Sciences},
  113(51):14552--14559, 2016.

\bibitem{Siet15}
C.~Siettos and G.~Bafas.
\newblock {\em Linear and nonlinear automatic control systems}.
\newblock Kallipos University press, 2015.

\bibitem{Son98}
E.~Sontag.
\newblock {\em Mathematical Control Theory, Deterministic Finite Dimensional
  Systems}.
\newblock Springer, 1998.

\bibitem{golubitsky2003symmetry}
M.~Golubitsky and I.~Stewart.
\newblock {\em The symmetry perspective: from equilibrium to chaos in phase
  space and physical space}, volume 200.
\newblock Springer Science \& Business Media, 2003.

\bibitem{Nic77}
G.~Nicolis and I.~Prigogine.
\newblock {\em Self-organization in Nonequilibrium Systems: From Dissipative
  Structures to Order Through Fluctuations citation}.
\newblock John Wiley \& Sons, 1977.

\bibitem{Matcont}
A.~Dhooge, W.~Govaerts, and Yu.~A. Kuznetsov.
\newblock Matcont: A matlab package for numerical bifurcation analysis of odes.
\newblock 29(2):141–164, 2003.

\bibitem{matcont2}
W.~Govaerts, Yu.~A. Kuznetsov, and H.~Meije.
\newblock Matcont, numerical bifurcation analysis toolbox in matlab.
\newblock \url{https://sourceforge.net/projects/matcont/}, 2022.

\bibitem{dillon1994pattern}
R.~Dillon, P.~K. Maini, and H.~G. Othmer.
\newblock Pattern formation in generalized turing systems.
\newblock {\em Journal of Mathematical Biology}, 32(4):345--393, 1994.

\bibitem{woolley2022boundary}
T.~E. Woolley.
\newblock Boundary conditions cause different generic bifurcation structures in
  turing systems.
\newblock {\em Bulletin of Mathematical Biology}, 84(9):1--38, 2022.

\bibitem{yang2004symmetric}
L.~Yang and I.~R. Epstein.
\newblock Symmetric, asymmetric, and antiphase turing patterns in a model
  system with two identical coupled layers.
\newblock {\em Physical Review E}, 69(2):026211, 2004.

\end{thebibliography}






\end{document}